\documentclass[a4paper]{amsproc}
\usepackage{amssymb}

\theoremstyle{plain}
 \newtheorem{theorem}{Theorem}[section]
 \newtheorem{lemma}{Lemma}[section]
 
\theoremstyle{definition}
 \newtheorem{definition}{Definition}[section]

\renewcommand{\le}{\leqslant}
\renewcommand{\ge}{\geqslant}

\title[Complex Iterations and Bounded Analytic Hyper-Operators]{Complex Iterations and Bounded Analytic Hyper-Operators}



\author[James Nixon]{\bfseries James Nixon}

\address{
Toronto\\
Canada}
\email{james.nixon@mail.utoronto.ca}

\newcommand{\G}{\Gamma}
\newcommand{\D}[2]{\frac{d^{#2}}{{d#1}^{#2}}}
\newcommand{\A}{\mathcal{I}}

\newcommand{\up}{\uparrow}

\begin{document}

\vspace{18mm} \setcounter{page}{1} \thispagestyle{empty}

\begin{abstract}
We give a method to solving the problem of iterating holomorphic functions to fractional or complex heights. We construct an auxiliary function from natural iterates of a holomorphic function; the auxiliary function will be differintegrable and the complex derivatives of the auxiliary function are the complex iterates of the original holomorphic function. We use Ramanujan's master theorem as a foundation and apply elementary theorems from complex analysis to arrive at our result. We provide non-trivial examples of holomorphic functions iterated to complex heights using these methods. We derive a closed form expression for what we call bounded analytic hyper-operators $\alpha \up^n x$ defined for $\alpha \in (1,e^{1/e})$, $x\in\mathbb{R}^+$ and $n\ge0$. These hyper-operators share the same recursive structure as the hyper-operators defined on the natural numbers but are instead analytic. They form a sequence of operators beginning with addition, multiplication, and exponentiation. Surprisingly these hyper-operators are bounded by $e$ as they grow on the real line for $n \ge 2$. We maintain an elementary yet very general discussion of the problem, as our solutions are specific instances of more general cases.
\end{abstract}

\maketitle

\section{Introduction}\label{sec1}

\setcounter{section}{1}
\setcounter{equation}{0}\setcounter{theorem}{0}

We begin with a brief explanation of the problem we are examining. The problem we speak of is intuitive and simple to state, but has proven to be a difficult one to solve. It has arisen in many contexts throughout mathematical history in the past hundred years, and many mathematicians have devoted work towards it or similar veins, including Schr$\ddot{o}$der \cite{ref2} and Ramanujan \cite{ref3}. The question arises from a study of the composition operator, holomorphic functions and iterates of a function under composition. The subject of complex dynamics is closely related. 

Taking care to notice that composition between functions behaves less simply than multiplication between numbers, the question can be better understood if we think of composition as multiplication and the number $e$ as a function. We can always find the square root under multiplication, $\sqrt{e}=e^{1/2}$ such that $\sqrt{e}\cdot\sqrt{e} = e$. We can also find $\sqrt[3]{e}=e^{1/3}$ such that $\sqrt[3]{e}\cdot\sqrt[3]{e}\cdot\sqrt[3]{e} = e$. Further, we can construct a holomorphic function $e^z$ where $z$ ``counts'' how many times we iterate multiplication of $e$. We ask, is it possible to do this if we replace multiplication with composition, and $e$ with a function? Can we find a holomorphic function where $z$ ``counts'' how many times we iterate composition of another function? 

To say this more mathematically, consider a holomorphic function $\phi$ sending open $G \to G$. Is it possible to construct a holomorphic function $g:G\to G$ the composite square root of $\phi$, $g  = \sqrt{\phi}$, such that $g(g(\xi)) = \phi(\xi)$? Or is it possible to construct a holomorphic function $h: G\to G$ the composite cube root of $\phi$, $h = \sqrt[3]{\phi}$, such that $h(h(h(\xi))) = \phi(\xi)$? Or in general, any function $f$ the composite $n$'th root, $f = \sqrt[n]{\phi}$, such that $(f \circ f \circ...(n\,times)... \circ f)(\xi) = \phi(\xi)$? The question can be generalized further if we write these functions as composite exponents of $\phi$. Let us say that $g = \phi^{\circ 1/2},\,h = \phi^{\circ 1/3},\,f = \phi^{\circ 1/n}$. Lets presume this concept can be extended to all $z \in \mathbb{C}$ not just $z = 1/k$ for $k$ a natural. Like exponentiation, can we generate a function $\phi^{\circ z}:G \to G$ holomorphic in $z$ on some open $\Omega$ and in $\xi$ on $G$ (finding a holomorphic function $e^z$). This function would satisfy the exponent laws $\phi^{\circ z_1} \circ \phi^{\circ z_2} = \phi^{\circ z_1 + z_2}$ and would equal the original function at $z =1$ ($e^{z_1}\cdot e^{z_2} = e^{z_1 + z_2}$ and $e^1 = e$). Let us appeal to a rigorous definition of a ``complex iteration'' of a holomorphic function.

\begin{definition}\label{df1}
Suppose $G,\Omega \subseteq \mathbb{C}$ are open with $1 \in \Omega$ and $\Omega$ closed under addition. Let the function $\phi(z,\xi):\Omega \times G\to G$ be holomorphic in $z$ and $\xi$. For $z_1, z_2 \in \Omega$, $\phi(z_1, \phi(z_2, \xi)) = \phi(z_1+z_2,\xi)$. We say $\phi(z,\xi)$ is a complex iteration of $\phi(1,\xi)$ on $G$.
\end{definition}

The reader may care to notice this definition is a modified definition of \emph{flow} in engineering and physics. By this we mean to say our concept of complex iteration is a close sister to flow. Flow, however, appears in vector analysis and takes only real exponents of iteration. We have the desire to be able to produce $\phi^{\circ z}$ where $z \in \mathbb{C}$ and not just $\phi^{\circ t}$ for $t \in \mathbb{R}$, which is why we generalize the concept of flow. Also, the term flow usually arises in practical applications of fluid dynamics, and we stray from these real world applications; we write solely in the realm of pure mathematics.

As we've defined a complex iteration of a holomorphic function it is not unique. There may be many different candidate functions which satisfy these properties. We do however have a uniqueness criterion. If a complex iteration can be generated using the methods we will construct, it is the only complex iteration that can be generated in such a way. This is a consequence of Ramanujan's master theorem, or more audaciously Carlson's theorem. This states that any two functions that are appropriately exponentially bounded and equal on the positive integers must equal everywhere. Since we will always be dealing with holomorphic functions bounded as such, these theorems come in use frequently throughout. 

Off-hand the problem is intimidating, composition is an operator of increased complexity in comparison to one like multiplication which follows more convenient and accessible laws. In constrast to this though, we have a strong way of analyzing these problems, and the methods we shall introduce are part of a larger framework of techniques useful in problems of iteration or recursion of holomorphic functions such as this. The main note that shall hilite the generality of these techniques is that a function that interpolates the natural iterates of a function $\phi$ and is exponentially bounded to our specifications is necessarily a complex iteration of $\phi$. This gives way as a surprising result that can be shown with a simple exercise in contour integration and a fast study of the Mellin transform. We will not write these proofs, despite their importance we have condensed them into a summary of what we use.

In the late nineteenth century, Schr$\ddot{o}$der proposed the existence of a type of function that reduced our problem of iteration into a simpler one--the calculation of the Schr$\ddot{o}$der function and its inverse. The Schr$\ddot{o}$der function was defined as an eigenfunction to the linear operator $\mathcal{C}_\phi$, where $\mathcal{C}_\phi f = f \circ \phi$. We will only need a local definition of the Schr$\ddot{o}$der function, for our purposes we reference one version of it from ~\cite{ref2}. It is properly called Koenigs linearization theorem.

\begin{theorem}[Koenig's Linearization Theorem]\label{lma1}
Suppose $\phi(\xi)$ is a holomorphic function on open $G$, for $\xi_0 \in G$ $\phi(\xi_0) = \xi_0$ and $0<|\phi'(\xi_0)|<1$. There exists a neighborhood $U$ about $\xi_0$ and a neighborhood $D$ about zero so that a holomorphic function $\Psi: U \to D$. Further $\Psi(\phi(\xi)) = \phi'(\xi_0) \Psi(\xi)$.
\end{theorem}

Although these functions are named for Schr$\ddot{o}$der, he did not produce any general method of retrieving the Schr$\ddot{o}$der function of an arbitrary holomorphic function. The first construction was due to Koenigs \cite{ref2}. In our investigations we will bypass the evaluation and rigorous construction of Schr$\ddot{o}$der functions and will instead use the conditions for existence. We can show that these conditions are sufficient for different evaluation and construction techniques to come into play. We see a definition right away of the complex iteration of $\phi$ using the Schr$\ddot{o}$der function. Define $\phi^{\circ z}(\xi) = \Psi^{-1}(\phi'(\xi_0)^z \Psi(\xi))$. This function formally satisfies our conditions, it is holomorphic and the recursion is satisfied. We cannot be satisfied with this however, despite that it satisfies our conditions. $\Psi$ will only be locally defined and so this definition applies only for $\xi$ in a tiny radius about $\xi_0$, which is quite an unsatisfactory result.

This leads us to our auxiliary function. We will lift the definition of $\phi^{\circ z}(\xi)$ from a tiny radius about $\xi_0$ to a much larger domain by using techniques from analysis. This function will always be denoted with $\vartheta$ and will be holomorphic in two variables. We say $\vartheta(w,\xi):\mathbb{C} \times G \to \mathbb{C}$ where $\vartheta(w,\xi) = \sum_{n=0}^\infty \phi^{\circ n+1}(\xi) \frac{w^n}{n!}$; using the notation $\phi^{\circ n} (\xi) = (\phi \circ \phi \circ\,...(n\,times)...\,\circ \phi)(\xi)$. 

To specify why we invent the function $\vartheta$, it is very similar to the motivation for the Schr$\ddot{o}$der function. It is here that we encode a more general idea. This function satisfies a very important differential equation. The identity $\D{w}{}\vartheta(w,\xi) = \mathcal{C}_\phi \vartheta(w,\xi) = \vartheta(w,\phi(\xi))$. Noticeably, just like how the Schr$\ddot{o}$der function sends composition by $\phi$ into multiplication by $\phi'(0)$; $\vartheta$ takes composition by $\phi$ and gives differentiation by $w$. This is the useful part of $\vartheta$ we care most about. In a more general setting we could encode a different linear operator as opposed to $\mathcal{C}_\phi$ (composition by $\phi$) through a similar auxiliary function, and it will satisfy a similar differential equation.

Noting this property of $\vartheta$, we further see that $\D{w}{k}\Big{|}_{w=0}\vartheta(w,\xi) = \phi^{\circ k+1}(\xi)$. The derivatives about zero are $\phi$'s iterates. We ask then, if we can take the complex derivatives of $\vartheta$ about zero, will this give $\phi$'s complex iterates?  The answer is yes. This requires some work in complex analysis however, especially when we look at the domain $G$ that $\xi$ lives in. In general this will work for complex iterating many more operators than just $\mathcal{C}_\phi$. If $E$ is a linear operator and we construct $\vartheta_E$ from $E$ how we constructed $\vartheta$ from $\mathcal{C}_\phi$ (through its iterates), we will find $\D{w}{s} \vartheta_E= E^s \vartheta_E $ and $E^s E^{s'} \vartheta_E = \D{w}{s} \D{w}{s'} \vartheta_E = \D{w}{s+s'} \vartheta_E = E^{s+s'}\vartheta_E$.

Our techniques will be developed by modifying the differintegral defined as the Riemann-Liouville differintegral operator when the lower limit is set to negative infinity. This operator is a modified Mellin transform, and appears frequently in complex analysis. We will gather some basic theorems on the differintegral. The first of such being a consequence of Ramanujan's master theorem. We also will provide an analytic continuation of our differintegral for suitable holomorphic functions. We will then state the ``factorization'' lemma. This lemma will prove to go very far in evaluating the complex iterates of certain holomorphic functions. The result is derived in complex analysis and requires little mention of the differintegral--however the efficacy of the notation and intuition from fractional calculus proves valuable. We proceed from this by providing an expression for the complex iterates of a holomorphic function by taking the differintegral of $\vartheta$ our auxiliary function. 

We provide non-trivial examples of these iteration methods at work. In doing such we introduce the bounded analytic hyper-operators. These are a sequence of analytic functions that when one of them is iterated it forms the next function in the sequence. They will be bounded and will satisfy a recursive structure isomorphic to the usual hyper-operators defined on the natural numbers. To understand these bounded analytic hyper-operators more intuitively we must first understand hyper-operators on their own. 

Hyper-operators are a sequence of binary operators defined on the natural numbers. In order to construct them, we start with successorship, the first hyper-operator in the sequence. Definitively addition is iterated successorship $a+1+1+...(b\,times)...+1 = a + b: \mathbb{N} \times \mathbb{N} \to \mathbb{N}$. To get the next operator in the sequence, iterate addition giving multiplication $a \cdot b = a + a +a ...(b\,times)...+a : \mathbb{N} \times \mathbb{N} \to \mathbb{N}$. We then iterate multiplication and we get exponentiation $a^b = a \up b = a \cdot a \cdots (b\,times)\cdots a$, the next operator. Iterating this we get tetration $a \up^2 b = a \up a \up \dots (b\,times) \dots \up a$. Iterate tetration to get pentation $a \up^3 b$. So on and so forth. These operators can be formally defined as $a \up^n 1 = a$ and $a \up^{n} (a \up^{n+1} b) = a \up^{n+1} (b+1)$. This can be written more suggestively as $a \up^{n+1} b = a \up^n a \up^n ...(b\,times)...\up^n a$. It is not difficult to see that these operators begin to grow astronomically fast as operations on the natural numbers as we increase $n$. They become impossible to compute as the numbers get so large. It would take a very long time for a computer to even calculate $3 \up^3 3$.

Hyper-operators have been used to prove various properties about the natural numbers. For example they were used by Ackermann to construct the Ackermann function, which was the first constructed function to not be primitive recursive. They were used to induce a hierarchy on the asymptotic rate of growth of primitive recursive functions. Each operator gives way to a new class of functions that grow faster than the last class. They also express a natural recursive continuation to the three operators we hold most important: addition, multiplication and exponentiation.

We shall not look into the natural numbers and hyper-operators--we will only gather our intuition from them. We will construct a sequence of analytic functions that satisfy the same recursive structure as hyper operators. In such a sense we find a solution to $\alpha \up^n x$ for $\alpha \in (1,e^{1/e})$, $x \in \mathbb{R}^+$ and $n \in \mathbb{N}$. This function will satisfy the recursion $\alpha \up^n(\alpha \up^{n+1} x) = \alpha \up^{n+1} (x+1)$; wherein $\alpha \up^0 x = \alpha \cdot x$ and $\alpha \up x = \alpha^x$ are unbounded and $\alpha \up^n x$ is bounded by a number less than or equal to $e$ for $n\ge2$. These functions will be analytically continued from $\alpha \up^n:\mathbb{R}^+ \to (1,e)$ to a function $\alpha \up^n : \mathbb{C}_{\Re(z)>0} \to \mathbb{C}$ but we have no knowledge of where these functions send to; they behave quite chaotically. It is with a slight abuse of notation that we write $\alpha \up^n x$. The up-arrow notation was originally devised by Knuth in \cite{ref4} and defined only on natural numbers. We use this notation solely to express that the nested structure is the same as the usual hyper-operators. In the end we are given a condensed theorem:

\begin{theorem}
Let $1 \le \alpha \le e^{1/e}$ and $n\ge0$. Define the following holomorphic functions recursively for $z,w \in \mathbb{C}$, $\Re(z) > 0$ and $n \in \mathbb{N}$ with $\alpha \up^0 z = \alpha\cdot z$,
\begin{eqnarray*}
\vartheta_n(w) &=& \sum_{k=0}^\infty \big{(}\alpha\up^{n}\alpha\up^{n}...(k+1)\,times...\up^{n}\alpha\big{)} \frac{w^k}{k!}\\
\alpha \up^{n+1} z &=& \D{w}{z-1}\Big{|}_{w=0} \vartheta_n(w)
\end{eqnarray*}

then,
\begin{enumerate}
\item $\alpha \up^{n} :\mathbb{R}^+ \to \mathbb{R}^+$ and $\alpha \up^n (\alpha \up^{n+1} x) = \alpha \up^{n+1}(x+1)$
\item $\alpha \up^{n} x$ is real analytic in $\alpha$ for $1<\alpha < e^{1/e}$
\item $\D{x}{}\alpha \up^{n} x \ge 0$
\item $\alpha \up^n 0^+ = 1$ for $n\ge 1$
\end{enumerate}
\end{theorem}

 Bounded analytic hyper-operators are an intellectual curiousity. They exhibit a rather wild recursion and require a nonstandard approach in their solution. We produce its solution to shed light on the types of holomorphic functions we can iterate and the types of problems in recursion we can solve with transforms from fractional calculus. It is also an interesting puzzle to answer: does there exist analytic functions that satisfy the same recursion as hyper-operators on the naturals? The answer is yes and we can produce infinitely many through similar techniques, however instead of drawing a long general theorem we motivate the method with an example. Much like how we motivate iterating linear operators through the iteration of $\mathcal{C}_\phi$.

We extend an eye to the simplicity of the methods we use in totality. Nothing in this paper extends far from an undergraduate's knowledge in mathematics excepting perhaps a few theorems in complex analysis. The fractional calculus we use is so little that the phraseology could have been chosen to not mention it at all. However, we found our motivation for these results in fractional calculus and some of the intuitive ideas are clearer using our differintegral (the modified Mellin transform) rather than just the Mellin transform. We have attempted to be as simple as possible in effort only to display the rather small leaps in logic that allow us these constructions.

\section{Properties of the differintegral}\label{sec2}

\setcounter{section}{2}
\setcounter{equation}{0}\setcounter{theorem}{0}

We provide a brief expository on our differintegral in this section. This differintegral needs little study for our purposes and the properties we use of it are simple and require very little mention of fractional calculus. The following lemmas are valuable assets and can be applied in more general settings. For convenience we will restrict the differintegral to specific entire functions centered about zero. 

We assume the reader is somewhat familiar with the inverse Gamma function. We define it by its product representation, giving us an entire function \cite{ref1}. This function will appear in almost all of our equations and appears as frequently as $\frac{1}{2\pi i}$ does in complex analysis. We need not analyze the inverse Gamma function in depth, it simply appears everywhere in our formulas and normalizes our equations. We take our definition from \cite{ref1}, noting $\gamma = \lim_{n\to\infty} \sum_{j=1}^n\frac{1}{j} - \log(n)$,

$$\frac{1}{\G(z)} = ze^{\gamma z} \prod_{k=1}^\infty (1+\frac{z}{k}) e^{-\frac{z}{k}}$$

We define a restricted form of the differintegral. It is defined this way as to minimalize the amount of words and terms we need add later in the theorem. It makes sense to fix some of the parameters of the usual differintegral so that we have clearer explanations.

\begin{definition}\label{df3}
Let $f(w)$ be entire. For $\sigma \in \mathbb{R}^+\,0<\sigma < 1$ and $\theta,\kappa \in \mathbb{R} \,\, |\theta| < \kappa < \pi/2$ for some $\kappa$ let $\int_0^\infty |f(-e^{i\theta}t)|t^{-\sigma} \,dt < \infty$. Then $f$ is differintegrable and the differintegral centered at zero $\D{w}{-z}\Big{|}_{w=0}$ of $f$ is defined for $0 < \Re(z) < 1$,
\begin{equation}
\label{TWeyl}
\D{w}{-z}\Big{|}_{w=0} f(w) = \frac{e^{i\theta z}}{\Gamma(z)} \int_{0}^\infty f(-e^{i\theta}t)t^{z-1}\,dt
\end{equation}
\end{definition}

We note that $\D{w}{-z}\Big{|}_{w=0} f(w)$ is holomorphic in $z$. This follows because the integral expression $\int_0^n f(-e^{i\theta}t)t^{z-1}\,dt$ uniformly converges in $z$ as $n\to\infty$. We also have 

\begin{equation}\label{eq:2}
|\D{w}{-z}\Big{|}_{w=0} f(w)| < e^{(\pi/2 - \kappa)|\Im(z)|} \Leftrightarrow \int_0^\infty |f(-e^{i\theta}t)|t^{\sigma-1} \,dt < \infty\,\,|\theta| < \kappa
\end{equation}
 which is not difficult to show and is left as an aside.

The following lemma, which in its more general form is commonly referred to as Ramanujan's master theorem, is usually phrased in terms of the Mellin transform, however we shall phrase it in terms of our differintegral. The theorem is presented in its more general form in \cite{ref3}, we will restrict the theorem for our applications.

\begin{theorem}[Ramanujan's Master Theorem]\label{lm1}
Suppose $\phi(z)$ is holomorphic on $\mathbb{C}_{\Re(z) > 0}$. Assume that $|\phi(z)| < C e^{\kappa|\Im(z)| + \rho|\Re(z)|}$ for $\kappa, \rho, C \in \mathbb{R}^+$ and $\kappa < \pi/2$. The function $\vartheta(w) = \sum_{k=0}^\infty \phi(k+1) \frac{w^k}{k!}$ is entire in $w$, differintegrable and $\D{w}{-z}\Big{|}_{w=0} \vartheta(w)=\phi(1-z)$ for $0<\Re(z) < 1$.
\end{theorem}

We give an analytic continuation of our differintegral and the Mellin transform. Interestingly this definition gives way to a new representation of the Riemann-Liouville differintegral that is convergent for polynomials--however this is irrelevant to our discussion. This lifting of the differintegral from $-1<\Re(z) <0$ to $-1<\Re(z)$ is very beneficial computationally for us, and eases our methods.

\begin{lemma}\label{lm2}
 Suppose that $\vartheta(w) = \sum_{k=0}^\infty a_k \frac{w^k}{k!}$ is differintegrable, $\D{w}{z}\Big{|}_{w=0} \vartheta(w)$ can be analytically continued for all $\Re(z) > -1$
$$\D{w}{z}\Big{|}_{w=0} \vartheta(w) = \frac{1}{\G(-z)}\Big{(}\sum_{k=0}^\infty a_k \frac{(-1)^k}{k!(k-z)} + \int_1^\infty \vartheta(-w)w^{-z-1}\,dw\Big{)}$$
\end{lemma}

\begin{proof}

Break the differintegral into parts by taking $\int_0^\infty = \int_0^1 + \int_1^\infty$.  Then we know that:
\begin{eqnarray*}
\int_0^1 \vartheta(-w)w^{z-1}\,dw &=& \sum_{k=0}^\infty a_k\frac{(-1)^k}{k!} \int_0^1 w^{k+z-1}\,dw\\
&=& \sum_{k=0}^\infty a_k\frac{(-1)^k}{k!(z+k)}
\end{eqnarray*}

The above steps are justified; $\vartheta$'s Taylor series has uniform convergence on all of $\mathbb{C}$ and so the integral can be taken through the sum. The series in $z$ is uniformly convergent on compact subsets of $\mathbb{C}_{\Re(z)<1}/\{0,-1,-2,-3,...,\}$ and defines a meromorphic function there. This is a quick exercise. This expression is holomorphic when we multiply it by the inverse $\G$ function; the simple zeroes of $1 / \G$ occur where the simple poles of $\int_0^1 f(-w) w^{z-1}\,dw$ occur. This implies that $\frac{1}{\G(z)}\int_0^1 f(-w)w^{z-1}\,dw = \frac{1}{\G(z)}\sum_{n=0}^\infty a_n\frac{(-1)^n}{n!(z+n)}$  is holomorphic for all $\Re(z) < 1$.

Observing $\int_1^\infty$ we know that $|\vartheta(-e^{i\theta}w)| < M/w$ for some constant $M\in \mathbb{R}^+$ with $w>1$. Take $N$ big enough so that for $n>N$ we have $\Re(z) < \sigma<1$, $\int_n^\infty w^{\sigma-2}\,dw < \epsilon/M$. Then $|\int_1^n \vartheta(-e^{i\theta}w)w^{z-1}\,dw - \int_1^\infty \vartheta(-e^{i\theta}w)w^{z-1}\,dw| < \int_n^\infty |\vartheta(-e^{i\theta}w)|w^{\sigma-1}\,dw < M \int_n^\infty w^{\sigma-2}\,dw < \epsilon$. This shows uniform convergence and hence the result.
\end{proof}

Combining these lemmas we have a factorization of certain holomorphic $\phi$ over its values on $\mathbb{N}$. This is the result in complex analysis we will make use of most. It plays a surprising role in producing the complex iteration of certain functions, and poses the most questions about generalizations. It is applicable in many other areas, and expresses the more general character of problems addressed in this paper.

\begin{lemma}\label{lm3}
Suppose $\phi(z)$ is holomorphic on $\mathbb{C}_{\Re(z)>0}$. Assume that $|\phi(z)| < C e^{\kappa|\Im(z)| + \rho|\Re(z)|}$ for $\kappa, \rho, C \in \mathbb{R}^+$ and $\kappa < \pi/2$, then $$\phi(z) = \D{w}{z-1}\Big{|}_{w=0} \sum_{k=0}^\infty \phi(k+1) \frac{w^k}{k!}$$ or written explicitly $$\phi(z) = \frac{1}{\G(1-z)}\Big{(}\sum_{k=0}^\infty \phi(k+1) \frac{(-1)^k}{k!(k+1-z)} + \int_1^\infty \big{(}\sum_{k=0}^\infty \phi(k+1) \frac{(-w)^k}{k!}\big{)}w^{-z}\,dw\Big{)}$$.
\end{lemma}

We add that this lemma implies functions $f,g$ defined on the right half plane, bounded as $\phi$ is in Lemma \ref{lm3}, with $f(n) = g(n)$ for $n \in \mathbb{N}\,\,n\ge 1$ satisfy $f(z) = g(z)$ for all $\Re(z) > 0$. Let us call functions such as $\phi$ stated in Lemma \ref{lm3} \emph{factorable functions} to specify what we mean more clearly. 

Finally a theorem on the difference between differintegrable functions and functions that are simply in $L_1$. The difference is striking and describes the important inclusion of the nasty term $e^{i\theta}$ in our definition of differintegrability.

\begin{theorem}\label{thmDiff}
Let $\vartheta(w)$ be an entire function. If $|\int_0^\infty \vartheta(-e^{i\theta} w)\,dw| < \infty$ for $|\theta| < \kappa < \pi/2$ then $\vartheta$ is differintegrable.
\end{theorem}

\begin{proof}
Take the contour $C_R$ which is the boundary of the set $S_{\theta,R} = \{w \in \mathbb{C} \, : \,0 < \arg(w) < \theta \, |w| < R\}$ which we allow to include $0$. Let $\tau > 0$. Observe that:

\begin{eqnarray*}
0 &=& \int_{C_R} e^{-\tau w}\vartheta(-w)\,dw\\
&=& \int_0^R e^{-\tau w} \vartheta(-w)\,dw - e^{i\theta}\int_0^R e^{-\tau e^{i\theta} w}\vartheta(-e^{i\theta}w)\,dw - \int_{0}^{\theta} e^{-\tau Re^{it}}\vartheta(-Re^{it})(iRe^{it})\,dt\\
\end{eqnarray*}

It follows that as $R \to \infty$ the last term tends to $0$ and:

$$\int_0^\infty e^{-\tau w} \vartheta(-w)\,dw = e^{i\theta}\int_0^\infty e^{-\tau e^{i\theta} w}\vartheta(-e^{i\theta}w)\,dw$$

Letting $\tau \to 0$ this implies:

$$\int_0^\infty \vartheta(-w)\,dw = e^{i\theta}\int_0^\infty \vartheta(-e^{i\theta}w)\,dw$$

which implies since 

$$0 = \int_0^R\vartheta(-w)\,dw - e^{i\theta}\int_0^R\vartheta(-e^{i\theta}w)\,dw - \int_{0}^{\theta} \vartheta(-Re^{it})(iRe^{it})\,dt$$

as the limit is taken, $R \to \infty$, we must have:

$$\int_{0}^{\theta} \vartheta(-Re^{it})(iRe^{it})\,dt \to 0$$

All in all this is a sufficient condition for $\vartheta(-e^{i\theta}w)w \to 0$ as $w \to \infty$ as the integral $\int_{0}^{\theta}\vartheta = 0$ for $0<\theta < \kappa$. Equally so this result can be shown for $-\kappa < \theta < 0$.

In total this is a sufficient condition for:

$$\int_0^\infty |\vartheta(-e^{i\theta}t)|t^{-\sigma}\,dt < \infty$$

when $0 < \sigma < 1$
\end{proof}

With these results in complex analysis we are prepared to solve the problem of producing a complex iteration of a holomorphic function.

\section{Complex iterations of holomorphic functions}\label{sec3}

\setcounter{section}{3}
\setcounter{equation}{0}\setcounter{theorem}{0}

We start with the lemma that was given in the first section. We will use the existence of a Schr$\ddot{o}$der function to define a complex iteration of certain $\phi$ in a tiny neighbourhood of a fixed point. This method will only work for fixed points whose multiplier (the functions derivative at the fixed point) is real positive and between zero and one. We then proceed by \emph{factoring} this function. This will give us an expression for the complex iterations in a tiny area about the fixed point, but we will have no mention of the Schr$\ddot{o}$der function. We only needed its existence to produce the result, computationally everything works out independent of the evaluation of the Schr$\ddot{o}$der function.

The techniques we use will apply on more functions than these, but it will stray off topic from our goal of bounded analytic hyper-operators. There is nothing complicated about the following lemma, however it plays a deceptively important role. In fact if we were guaranteed this lemma on more linear operators than just $\mathcal{C}_\phi$ then this entire section could be shown on a much broader scale.
 
\begin{lemma}\label{lma4}
Let $\phi$ be holomorphic on open $G$. Assume $\phi$ fixes the point $\xi_0$ and that $0<\phi'(\xi_0) <1$. There exists $B\subseteq G$, simply connected and open, with $\xi_0 \in B$ such that the complex iteration $\phi^{\circ z}(\xi) : \mathbb{C}_{\Re(z) > 0} \times B \to B$, where
$$\phi^{\circ z} (\xi) = \D{w}{z-1}\Big{|}_{w=0} \sum_{n=0}^\infty \phi^{\circ n+1}(\xi) \frac{w^n}{n!}$$
\end{lemma}

\begin{proof}
Take our Sch$\ddot{o}$der function $\Psi$ which sends a neighborhood $U$ about $\xi_0$ to a neighborhood of zero. We know $\Psi(\phi(\xi)) = \lambda \Psi(\xi)$ for $0<\lambda = \phi'(\xi_0) < 1$ and the neighborhood about zero can be chosen arbitrarily small so that $\Psi$ is injective on this neighborhood. Define $\phi^{\circ z}(\xi) = \Psi^{-1}(\lambda^z\Psi(\xi))$. We want $\xi$ to live in the preimage of $\Psi$ about the largest circle possible so that $\lambda^z$ just moves us around inside this circle. Let $\xi \in \Psi^{-1}(D_\tau) = B$, where $D_\tau$ is the largest disk of radius $\tau$ about zero inside of $\Psi(U)$. Then $|\lambda^z \Psi(\xi)| < |\Psi(\xi)| < \tau$ and therefore $\phi^{\circ z}(\xi) = \Psi^{-1}(\lambda^z \Psi(\xi))$ defines a holomorphic function for $\xi \in B$ and $z \in \mathbb{C}_{\Re(z)>0}$ that sends to $B$. Interestingly, as the exponential function winds around the unit disk, $\phi^{\circ z}$ winds around $B$.

Carefully analyzing the definition of the complex iteration of $\phi$ we see that $\phi^{\circ z}$ is periodic in $z$ with period $2\pi i / \ln(\lambda)$. Thus it is bounded as the imaginary argument grows. Also $\phi^{\circ z}(\xi) \to \xi_0$ as $\Re(z) \to \infty$ and therefore it is bounded as the real argument grows. We are allowed to factor $\phi^{\circ z}(\xi)$. Let $\vartheta(w,\xi) = \sum_{n=0}^\infty \phi^{\circ n+1}(\xi) \frac{w^n}{n!}$ and we have that

$$\phi^{\circ z}(\xi) = \D{w}{z-1}\Big{|}_{w=0}\vartheta(w,\xi)$$

Further $\int_0^\infty |\vartheta(-e^{i\theta}w,\xi)|w^{\sigma} \, dw < \infty$ for $|\theta| < \pi/2$ and $0<\sigma<1$.

\end{proof}

We have an expression for the complex iteration independent of the Schr$\ddot{o}$der function. We now face the problem of finding a larger domain in $\xi$ by which our iterate $\phi^{\circ z}(\xi)$ makes sense. As of this point we have only devised a solution for $\phi^{\circ z}(\xi_0 + \delta)$, where $\xi_0$ is a certain type of fixed point and $|\delta| < \epsilon$ is very small. We will see however, that due to the behaviour of our auxiliary function and our differintegral, we can prove convergence on a larger non-trivial area so long as convergence is guaranteed in some small simply connected set about $\xi_0$. We will maximize the domain this iterate can be defined on. In order to do this we will take an object from complex dynamics that appears frequently.

If we define the basin of attraction: $\mathcal{A} = \{ \xi \in \mathbb{C}\,| \lim_{n\to\infty} \phi^{\circ n}(\xi) \to \xi_0\}$ for the function $\phi$ with fixed point $\xi_0$ such that $0< |\phi'(\xi_0)| < 1$, $\mathcal{A}$ is open \cite{ref2}. This is an argument that we will make use of frequently, stated explicitly: for each $\xi$ such that $\phi^{\circ n}(\xi) \to \xi_0$ as $n\to \infty$, there is an open ball about $\xi$ such that for $\zeta$ in this ball $\phi^{\circ n}(\zeta) \to \xi_0$ as $n\to\infty$. We will not use anything other than this property from the basin of attraction. We shall instead be concerned with the immediate basin of attraction which is a maximal connected subset of the basin of attraction. 

\begin{definition}\label{def4}
Take the holomorphic function $\phi:G\to G$ open. Assume $\phi$ fixes $\xi_0$ and that $0<|\phi'(\xi_0)|< 1$. The immediate basin of attraction $\A_{\xi_0} \subseteq G$ is the largest connected domain about $\xi_0$ such that $\phi^{\circ n}(\xi) \to \xi_0$ as $n \to \infty$. 
\end{definition}

$\A_{\xi_0}$ is open \cite{ref2}. As an important note $\A_{\xi_0}$ is maximal. This gives a quick proof that $\phi(\A_{\xi_0}) \subset \A_{\xi_0}$. $\A_{\xi_0}$ is the largest open connected set inside of $G$ satisfying $\lim_{n\to\infty} \phi^{\circ n}(\xi) \to \xi_0$ and $\xi_0 \in \A_{\xi_0}$. We know that $\phi(\A_{\xi_0})$ is connected and open and a subset of $G$, and since $\phi(\xi_0) = \xi_0$ we know $\xi_0 \in \phi(\A_{\xi_0})$. We also have that the limit converges to the constant $\xi_0$ as iterates of $\phi$ are applied, $\lim_{n\to \infty} \phi^{\circ n}( \xi') \to \xi_0$ for $\xi' \in \phi(\A_{\xi_0})$, therefore $\phi(\A_{\xi_0}) \subseteq \A_{\xi_0}$. 
 
We begin with a lemma determining convergence of the auxiliary function $\vartheta$ on a larger domain in $\xi$. We then move to a lemma that allows us to rearrange the infinite series representing the auxiliary function. This will reveal some hidden structure about the function and allow us to rephrase the question of differintegrability into a question on the convergence of a particular series representing the auxiliary function. Then similarly as for the first lemma of this section, the differintegral of the auxiliary function about zero will be the complex iterates of $\phi$. These will be defined on a larger domain in $\xi$, namely $\A_{\xi_0}$. This maximal subset will suffice for our work on bounded analytic hyper-operators.

\begin{lemma}\label{lm4}
Let $\phi: G \to G$ be a holomorphic function on $G$ open. Assume that $\xi_0 \in G$ satisfies $\phi(\xi_0) = \xi_0$ and $0<\phi'(\xi_0) < 1$. The auxiliary function $\vartheta$ converges uniformly for all $w \in \mathbb{C}$ and $\xi \in \A_{\xi_0}$, 
\begin{equation}\label{eq:1}
\vartheta(w,\xi) = \sum_{n=0}^\infty \phi^{\circ n+1}(\xi) \frac{w^n}{n!}
\end{equation}
\end{lemma}

\begin{proof}
Observe $\phi^{\circ n}:\A_{\xi_0}\to \A_{\xi_0}$ and $\phi^{\circ n} \to \xi_0$ as $n$ grows. Take $\Omega \subset \A_{\xi_0}$ compact and the sequence $M_n = \sup_{\xi \in \Omega} \{|\phi^{\circ n+1}(\xi)|\}$ which converges to $|\xi_0|$. Therefore the sequence is bounded, $M_n < M$ for some $M \in \mathbb{R}^+$. Define the partial sums of $\vartheta$ such that $\vartheta_N(w,\xi) = \sum_{n=0}^N \phi^{\circ n+1}(\xi) \frac{w^n}{n!}$ and choose $N$ big enough such that $\sum_{n=N+1}^\infty \frac{R^n}{n!} < \frac{\epsilon}{M}$. Therefore $|\vartheta(w,\xi)-\vartheta_N(w,\xi)| \le \sum_{n=N+1}^\infty |\phi^{\circ n}(\xi)| \frac{|w|^n}{n!} < \sum_{n=N+1}^\infty M \frac{R^n}{n!} < \epsilon$ where we bounded $|w|<R$ in some compact $\Omega' \subset \mathbb{C}$. Since $\epsilon$ is arbitrary the series \eqref{eq:1} converge uniformly for $\xi \in \A_{\xi_0}$ and $w \in \mathbb{C}$.
\end{proof}

Following from this we note that by Weierstrass, since $\vartheta_N = \sum_{n=0}^N \phi^{\circ n+1}(\xi)\frac{w^n}{n!}$ converges uniformly to $\vartheta$ in $w$ and $\xi$ as $N\to\infty$, its derivatives in both variables converge uniformly when we fix one of them. This gives us a nice trick that we would not have had without Lemma \ref{lma4}. We will show that not only does $\vartheta(w,
\xi)$ have a convergent differintegral, but its derivatives in $\xi$ do as well, $\D{\xi}{k}\vartheta(w,\xi)$ has a convergent differintegral. We will do this for an arbitrary simply connected set $B$ that $\xi$ lives in. This gives us a more general lemma that will be used in an induction step when we attempt to perform the iterate in a larger domain. 

\begin{lemma}\label{lm5}
Let $\phi: G \to G$ be a holomorphic function on $G$ open. Assume that $\xi_0 \in G$ satisfies $\phi(\xi_0) = \xi_0$ and $0<\phi'(\xi_0) < 1$. Let $\vartheta(w,\xi) = \sum_{n=0}^\infty \phi^{\circ n+1}(\xi) \frac{w^n}{n!}$. For $\xi$ in some open $B \subseteq \A_{\xi_0}$ assume for $\sigma,\theta \in \mathbb{R}$  $0 < \sigma < 1$ $|\theta|<\pi/2$ that $\int_0^\infty | \vartheta(-e^{i\theta}w,\xi)| \,dw< \infty$. For $k \in \mathbb{N}$ we have $$\int_0^\infty |\D{\xi}{k}\vartheta(-e^{i\theta}w,\xi)|\,dw < \infty$$
\end{lemma}

\begin{proof}
We must show that the differintegral converges uniformly in $\xi$, so that $\frac{e^{i\theta z}}{\G(z)}\int_0^n \vartheta(-we^{i\theta},\xi)w^{z-1}\,dw \to \D{w}{-z}\Big{|}_{w=0} \vartheta(-w,\xi)$ uniformly in $\xi$ and $z$ as $n$ grows for $|\theta| < \pi/2$ and $0 < \Re(z)  < 1$, $\xi \in B$. We get that $\theta$ is bounded by $\pi/2$ because $|\phi^{\circ z}(\xi)| < M$ as $|\Im(z)| \to \infty$ and by \eqref{eq:2}. Take $\Omega \subset B$ a compact disk. Define the function in $w$, $g(w) = \sup_{\xi \in \Omega} \{|\vartheta(w,\xi)|\}$. Note that $g(-e^{i\theta}w)w \to 0 $ as $w \to \infty$.

By using Cauchy bounds on the compact disk $\Omega$ of radius $\rho$ we must have $|\D{\xi}{k}\vartheta(-w,\xi)| < \frac{k!}{\rho^k}g(w)$. And therefore $\D{\xi}{k} \vartheta(-w,\xi)w \to 0$ as $w \to \infty$.
\end{proof}

We justify rearrangement of $\vartheta(w,\xi)$ as a Taylor series in $\xi$ about any point $\zeta \in \A_{\xi_0}$. This will allow us to show convergence of $\vartheta$ in the differintegral by rearranging the expression representing $\vartheta$. The trick is slightly hidden but the result will be clearer when we tie these lemmas together. The convergence of the derivatives of $\vartheta$ implies elements in a disk about $\xi_0$ converge as well, so long as that disk still lives inside of $\A_{\xi_0}$. This disk will always exists because $\A_{\xi_0}$ is open.

\begin{lemma}\label{lm6}
Let $\phi: G \to G$ be holomorphic on $G$ open. Assume $\xi_0 \in G$ satisfies $\phi(\xi_0) = \xi_0$ and $0< \phi'(\xi_0) < 1$. Let $\vartheta(w,\xi) = \sum_{n=0}^\infty \phi^{\circ n+1}(\xi) \frac{w^n}{n!}$. Let $D_\zeta$ be a disk about $\zeta\in\A_{\xi_0}$ such that $D_{\zeta}\subset\A_{\xi_0}$. For $\xi \in D_{\zeta}$ and $w \in \mathbb{C}$, $\vartheta$ is represented by $\vartheta(w,\xi) = \sum_{k=0}^\infty \D{\xi}{k}\Big{|}_{\xi =\zeta} \vartheta(w,\xi)\frac{(\xi - \zeta)^k}{k!}$.
\end{lemma}

\begin{proof}
Take $\vartheta(w,\xi)$ and expand it into a Taylor series in $\xi$ about $\xi= \zeta$ ignoring convergence issues momentarily and expand this expression into a Taylor series in $w$. This will imply that $\vartheta$ equals the formal double series $$\vartheta(w,\xi) = \sum_{n=0}^\infty \sum_{k=0}^\infty \Big{(}\D{\xi}{k}\Big{|}_{\xi=\zeta}\phi^{\circ n+1}(\xi)\Big{)} \frac{w^n (\xi - \zeta)^k}{n!k!}$$

We show uniform convergence of the double series by using basic techniques in complex analysis. Get Cauchy bounds on the derivatives of $\phi^{\circ n+1}(\xi)$ in $\xi$ by contour integrating $\phi^{\circ n+1}(\xi)$ around $\partial D_{\zeta}$. This tells us that $|\D{\xi}{k}\Big{|}_{\xi=\zeta}\phi^{\circ n+1}(\xi)| < \frac{k!M_n}{\ell^k}$, for $\ell$ the radius of $D_{\zeta}$ and $M_n = \sup_{\xi \in \partial D_\zeta} \{|\phi^{\circ n+1}(\xi)|\}$. We know that this sequence converges $M_n \to |\xi_0|$, therefore it is bounded $M_n < M$ for some $M \in \mathbb{R}^+$. Thus, to show the double sum converges uniformly, take the partial sums $\vartheta_{N}(w,\xi) = \sum_{n=0}^\infty \sum_{k=0}^N \D{\xi}{k}\Big{|}_{\xi = \zeta}\phi^{\circ n+1}(\xi)\frac{w^n(\xi - \zeta)^k}{n!k!}$. $\vartheta_N$ is well defined for each $N$ by Lemma \ref{lm4}, the first infinite series is a uniformly convergent series in $\xi$, therefore so are its derivatives in $\xi$ and finite sums of them. Bound $|w| < R\in \mathbb{R}^+$ and take $N$ big enough so that $\sum_{k=N+1}^\infty \Big{(}\frac{L}{\ell}\Big{)}^k < \frac{\epsilon}{M}e^{-R}$, for $L \in \mathbb{R}^+$ such that $|\xi - \zeta|<L < \ell$. 

\begin{eqnarray*}
|\vartheta(w,\xi) - \vartheta_{N}(w,\xi)| &\le& \sum_{n=0}^\infty \sum_{k=N+1}^\infty \Big{|}\D{\xi}{k}\Big{|}_{\xi=\zeta}\phi^{\circ n+1}(\xi)\Big{|} \frac{|w|^n |\xi - \zeta|^k}{n!k!}\\
&<&\sum_{n=0}^\infty \sum_{k=N+1}^\infty M_n\frac{|w|^n |\xi - \zeta|^k}{\ell^k n!}\\
&<& \sum_{n=0}^\infty M_n\frac{|w|^n}{n!}\sum_{k=N+1}^\infty \frac{|\xi - \zeta|^k}{\ell^k}\\
&<& M\sum_{n=0}^\infty \frac{R^n}{n!}\sum_{k=N+1}^\infty \frac{L^k}{\ell^k}\\
&<& M e^R \sum_{k=N+1}^\infty \frac{L^k}{\ell^k}\\
&<& \epsilon
\end{eqnarray*}

Since $\epsilon$ was arbitrary this double series converges uniformly and absolutely, this justifies the rearrangement for $\xi \in D_{\zeta}$ and $w \in \mathbb{C}$,
$$\vartheta(w,\xi) = \sum_{k=0}^\infty \Big{(}\D{\xi}{k}\Big{|}_{\xi = \zeta} \vartheta(w,\xi)\Big{)}\frac{(\xi - \zeta)^k}{k!}$$
\end{proof}

We give an explanation of the proof we are going to provide in the following theorem. Expand $\vartheta$ as a Taylor series in $\xi_1 \in \A_{\xi_0}$ about $\xi_0 \in \A_{\xi_0}$, where we have shown the derivatives of $\vartheta$ at $\xi_0$ are functions in $w$ that are differintegrable (by Lemma \ref{lm5}). This will show $\vartheta(w,\xi_1)$ has a converging differintergral because the terms of its Taylor series as functions in $w$ do. Then, we rinse and repeat and take $\xi_2 \in \A_{\xi_0}$ such that $\xi_2$ is in a disk about $\xi_1$. We then show that $\vartheta(w,\xi_2)$ has a convergent differintegral through the same process--its Taylor coefficients about $\xi_1$ are functions in $w$ that are differintegrable. We then proceed by induction on the number of intersecting disks within $\A_{\xi_0}$ we have to chain from $\xi_0$ to get to $\xi$. This process terminates for all $\xi$, concluding the proof. We have placed most of the work in the lemmas and we will see that our theorem follows with only some effort.

\begin{theorem}\label{thm2}
Let $\phi: G \to G$ be a holomorphic function on $G$ open. Assume there is a $\xi_0 \in G$ that satisfies $\phi(\xi_0) = \xi_0$ and $0< \phi'(\xi_0) < 1$. Let $\vartheta(w,\xi) = \sum_{n=0}^\infty \phi^{\circ n+1}(\xi) \frac{w^n}{n!}$ for $w \in \mathbb{C}$ and $\xi \in \A_{\xi_0}$. The complex iterate $\phi^{\circ z}(\xi) :\mathbb{C}_{\Re(z) > 0} \times \A_{\xi_0} \to \A_{\xi_0}$ is given by
$$\phi^{\circ z}(\xi) = \D{w}{z-1}\Big{|}_{w=0} \vartheta(w,\xi)$$
\end{theorem}

\begin{proof}
By Lemma \ref{lm6} we know that $\vartheta$ can be rewritten as $$\vartheta(w,\xi)= \sum_{k=0}^\infty \Big{(}\D{\xi}{k}\Big{|}_{\xi = \xi_0} \vartheta(w,\xi)\Big{)}\frac{(\xi - \xi_0)^k}{k!}$$for $\xi$ in some disk $D_0$ about $\xi_0$ contained in $\A_{\xi_0}$. We also have by Lemma \ref{lm5} that $\int_0^\infty \Big{|}\D{\xi}{k}\Big{|}_{\xi=\xi_0}\vartheta(-we^{i\theta},\xi)\Big{|}w^{\sigma - 1}\,dw < \infty$ for $\sigma,\theta \in \mathbb{R}$, $0 < \sigma < 1$ and $|\theta| < \pi/2$. Define the partial sums of $\vartheta$ as $\vartheta_j(w,\xi) = \sum_{k=0}^j \Big{(}\D{\xi}{k}\Big{|}_{\xi = \xi_0}\vartheta(w,\xi)\Big{)} \frac{(\xi - \xi_0)^k}{k!}$. This shows that $\int_0^\infty |\vartheta_j(-e^{i\theta}w,\xi)|\,dw < \infty$ for all $j$. 

Observe the following identity. Since $\phi^{\circ z}(\xi) = \int_0^\infty \vartheta(-e^{i\theta} w,\xi)w^{-z} \, dw$ for $0 < \Re(z) < 1$, allow $z \to 0$ and we must have:

$$e^{i\theta}\int_0^\infty \vartheta(-e^{i\theta}w,\xi) \, dw = \xi$$

Taking $\vartheta_i$, observe that $e^{i\theta} \int_0^\infty \vartheta_j(-we^{i\theta},\xi)\,dw = \xi_0 + (\xi - \xi_0)$ for $j \ge 2$ and $\xi \in D_0$, an arbitrary disk inside of $\A_{\xi_0}$ about $\xi_0$. Therefore:

$$e^{i\theta} \int_0^\infty \lim_{j\to\infty}\vartheta_j(-we^{i\theta},\xi)\,dw = \xi_0 + (\xi - \xi_0)$$

by Theorem \ref{thmDiff} it must be for all $0 < \sigma <1$:

$$\int_0^\infty |\vartheta(-we^{i\theta},\xi)|w^{\sigma}\,dw < \infty$$

for all $\xi \in D_0$ which shows the base step of induction.

We will now imitate this method of proof on all elements of $\A_{\xi_0}$. Each lemma was proven so that the steps that follow come together piece by piece in simple procession of each other. We go by induction on how many intersecting open disks inside of $\A_{\xi_0}$ it takes to chain $\xi_0$ to $\xi$. The case where an element is inside a disk inside of $\A_{\xi_0}$ centered about $\xi_0$ is our induction hypothesis; we know it is differintegrable for such $\xi$. Let us assume that it takes $m$ links in the chain of open disks to get to $\xi_m$ from $\xi_0$ and that $\int_0^\infty |\vartheta(-we^{i\theta},\xi_m)|w^{\sigma-1}\,dw<\infty$. 

Take a disk $D_{m}$ centered about $\xi_m$ that is still within $\A_{\xi_0}$. For $k \in \mathbb{N}$ we know that $\int_0^\infty |\D{\xi}{k}\Big{|}_{\xi = \xi_m}\vartheta(-we^{i\theta},\xi) |\,dw<\infty$. This follows by Lemma \ref{lm5}. Expand $\vartheta$ into a Taylor series about $\xi_m$, and exactly as we knew for the base case we know that if $\xi_{m+1} \in D_{m}$ that

$$\vartheta(w,\xi_{m+1}) = \sum_{k=0}^\infty \D{\xi}{k}\Big{|}_{\xi = \xi_m}\vartheta(w,\xi) \frac{(\xi_{m+1} - \xi_m)^k}{k!}$$

which follows by Lemma \ref{lm6}. We know this expression has a convergent differintegral in $w$ which follows just as the base case did (the beginning of this proof). Therefore the induction process is complete. We know that every point $\xi \in \A_{\xi_0}$ has a finite chain of intersecting disks inside of $\A_{\xi_0}$ which connect it to $\xi_0$. We know that we have convergence for every member of $\A_{\xi_0}$. This shows for all $\Re(z) > 0$ and for all $\xi \in \A_{\xi_0}$,

$$\phi^{\circ z}(\xi) = \D{w}{z-1}\Big{|}_{w=0} \vartheta(w,\xi)$$ or written in explicit form with no mention of the differintegral:

$$\phi^{\circ z}(\xi) = \frac{1}{\G(1-z)} \Big{(}\sum_{n=0}^\infty \phi^{\circ n+1}(\xi)\frac{(-1)^n}{n!(n+1-z)} + \int_1^\infty \Big{(}\sum_{n=0}^\infty \phi^{\circ n+1}(\xi) \frac{(-w)^n}{n!}\Big{)} w^{-z}\,dw\Big{)}$$

This is an analytic continuation of the original expression from Lemma \ref{lma4}, however we are unsure of the domains it is defined on. We will show that $\phi^{\circ z} (\xi) : \mathbb{C}_{\Re(z) > 0} \times \A_{\xi_0} \to \A_{\xi_0}$. 

Take a compact and connected set $\Omega \subset \A_{\xi_0}$ such that $\xi_0 \in \Omega$ and $B \subset \Omega$ where $B$ is the simply connected region from Lemma \ref{lma4}. There exists an $N$ such that for $n>N$ and $\xi \in \Omega$ we have $\phi^{\circ n}(\xi) \in B$. Now, $\phi^{\circ z} : B \to B$ and $\phi^{\circ z}(\phi^{\circ n}(\xi')) = \phi^{\circ z+n}(\xi')$ for $\xi' \in B$. Therefore by analytic continuation and since $\phi^{\circ n}(\xi) \in B$, $\phi^{\circ z}(\phi^{\circ n}(\xi))=\phi^{\circ z+n}(\xi)$ for $\xi \in \Omega$. Therefore $\phi^{\circ z+n}:\Omega \to \A_{\xi_0}$. Notice similarly that for $\xi' \in B$ we have $\phi^{\circ n}(\phi^{\circ z}(\xi')) = \phi^{\circ z+n}(\xi')$ which implies by analytic continuation and since $\phi^{\circ z+n} : \Omega \to \A_{\xi_0}$ that $\phi^{\circ n}(\phi^{\circ z}): \Omega \to \A_{\xi_0}$. This tells us for $\zeta \in \phi^{\circ z}(\A_{\xi_0})$ we know $\phi^{\circ n} (\zeta) \to \xi_0$ as $n \to \infty$, we know $\xi_0 \in \phi^{\circ z}(\A_{\xi_0})$ and we know $\phi^{\circ z}(\A_{\xi_0})$ is connected and open. Therefore $\phi^{\circ z}(\A_{\xi_0}) \subseteq \A_{\xi_0}$ by the fact $\A_{\xi_0}$ is the maximal set which satisfies these properties.

To conclude the proof we show $\phi^{\circ z_1} (\phi^{\circ z_2}(\xi)) = \phi^{\circ z_1 + z_2}(\xi)$. First note that both functions are well defined. Because $\phi^{\circ z}(\xi) \to \xi_0$ as $\Re(z) \to \infty$ and $\phi^{\circ z}$ is bounded as $|\Im(z)| \to \infty$--both functions $\phi^{\circ z_1}(\phi^{\circ z_2})$ and $\phi^{\circ z_1 + z_2}$ can be factored in $z_1, z_2$. Then note that because $\phi^{\circ z_1} (\phi^{\circ z_2}(\xi)) = \phi^{\circ z_1 + z_2}(\xi)$ for $z_1,z_2 \in \mathbb{N}$ they equal for all $\Re(z_1),\Re(z_2) > 0$. 
\end{proof}

\section{On the solution of Tetration for bases $1 < \alpha < e^{1/e}$}\label{sec4}

\setcounter{section}{4}
\setcounter{equation}{0}\setcounter{theorem}{0}

In this section we discuss an applied problem we can solve using our differintegral and the techniques we have just developed. The problem is recent in history and is of considerable difficulty. We will reduce the difficulty using our recently developed tools and show how problems of this type can be handled with Theorem \ref{thm2}. This section is a warm up, and precursor to the next which generalizes the problem we face in this section.

In order to phrase the question we must encourage the definition of tetration. Take a positive real number $\alpha$ and consider building an exponential tower from it, indexing the number of steps in the tower like a sequence. The terms of the sequence would be $^0 \alpha = 1,\, ^1 \alpha = \alpha,\, ^2 \alpha = \alpha^\alpha,\, ^3 \alpha = \alpha^{\alpha^\alpha},\, ^4 \alpha = \alpha^{\alpha^{\alpha^\alpha}},\,...,\,^n \alpha = \alpha^{^{n-1}\alpha},\,...$. 

The question arises from this sequence: can we find a holomorphic function for $\Re(z) > 0$ such that $^{z+1} \alpha = \alpha^{^{z} \alpha}$ and interpolates such a sequence? The answer is yes, and there exists many non-unique ways of generating these tetration functions. In fact we can always take $f(z) = \,^{z + \theta(z)} \alpha$ for some one periodic function $\theta$ with $\theta(0) = 0$, and this will be another solution to tetration. To our advantage though, our solution is the only one that can be factored. Therefore for the bases of the tetration function we can solve for we have a uniqueness criterion.

\begin{definition}\label{df5}
A holomorphic function $\mathcal{F}:\mathbb{C}_{\Re(z) > 0}\to\mathbb{C}$ is a tetration function base $\alpha\in \mathbb{R}^+$ iff $\mathcal{F}(1) = \alpha$ and $\mathcal{F}(z+1) = \alpha^{F(z)}$.
\end{definition}

In order to devise a solution to tetration we need to generalize the problem slightly. We will focus on performing iterates of the function $\alpha^{\xi}$ around a real fixed point, where the complex iterates when $\xi = 1$ will be tetration. In order for an exponential function $f(\xi) = \alpha^\xi$ to have a fixed point $\xi_0$ such that $f'(\xi_0) \in \mathbb{R},\,0<f'(\xi_0) < 1$ we must restrict our base to $1 < \alpha < e^{1/e}$, since for each $1<\alpha<e^{1/e}$ there exists a $1 < \beta < e$ such that $f(\beta) = \beta$ and $0<f'(\beta)<1$.

\begin{lemma}\label{lm8}
If $\alpha,\beta \in \mathbb{R}^+$, $1< \beta < e$ and $\alpha = \beta^{1/\beta}$ then the entire function $f(\xi) = \alpha^{\xi}$ has a fixed point at $\xi = \beta$ and $0 < f'(\beta) = \ln\beta < 1$.
\end{lemma}

\begin{proof}
Plug in $f(\beta)$ and observe it fixes $\beta$. Take the derivative and observe it is $f'(\beta)= \ln\beta$, which is between zero and one since $1 < \beta < e$.
\end{proof}

From here we go to show that $1$ is in the immediate basin of attraction about $\beta$; that $1 \in \A_{\beta}$. This will imply $f$'s complex iterates are tetration through simple manipulation of the definitions.

\begin{lemma}\label{lm9}
If $\alpha,\beta \in \mathbb{R}^+$ and $1< \beta < e$ and $\alpha = \beta^{1/\beta}$ then for the entire function $f(\xi) = \alpha^{\xi}$ with fixed point $\xi = \beta$, $1\in \A_{\beta}$.
\end{lemma}

\begin{proof}
First note for all $1-\epsilon < x_0 < \beta$ it follows $ f^{\circ n} (x_0) \to\beta$ as $n \to \infty$. Elements in open disks in $\mathbb{C}$ about $x_0$ satisfies this as well. The value $\beta$ is connected to $1$ with open balls, of which all elements $f^{\circ n} \to \beta$. Since $\A_{\beta}$ is the maximal set to satisfy these properties this chain of open balls is in $\A_{\beta}$. Thus, $1 \in \A_\beta$.
\end{proof}

With these results we can state the unique solution to tetration in one clean formula.

\begin{theorem}\label{thm3}
If $\alpha \in \mathbb{R}^+$ and $1 < \alpha < e^{1/e}$ then tetration base $\alpha$ is given by,
$$^z \alpha = \frac{1}{\G(1-z)} \Big{(}\sum_{n=0}^\infty \,(^{n+1} \alpha )\frac{(-1)^n}{n!(n+1-z)} + \int_1^\infty \big{(}\sum_{n=0}^\infty \,(^{n+1} \alpha) \frac{(-w)^n}{n!} \big{)}w^{-z}\,dw\Big{)}$$
\end{theorem}

\begin{proof}
Appeal to Theorem \ref{thm2} and Lemma \ref{lm9} and Lemma \ref{lm8}. By this we mean, iterate the function $f(\xi) = \alpha^{\xi}=\beta^{\xi/\beta}$ with fixed point $\beta$ and for $\xi \in \A_{\beta}$ we have 
$$f^{\circ z}(\xi) = \frac{1}{\G(1-z)} \Big{(}\sum_{n=0}^\infty \,f^{\circ n+1}(\xi) \frac{(-1)^n}{n!(n+1-z)} + \int_1^\infty \big{(}\sum_{n=0}^\infty \,f^{\circ n+1}(\xi) \frac{(-w)^n}{n!} \big{)}w^{-z}\,dw\Big{)}$$
Setting $\xi =1$ is the solution of tetration since $f^{\circ 1}(1) = \alpha$ and $ f^{\circ z+1}(1) = f(f^{\circ z}(1)) = \alpha^{f^{\circ z}(1)}$.
\end{proof}

\section{Bounded Analytic Hyper-Operators}\label{sec5}

\setcounter{section}{5}
\setcounter{equation}{0}\setcounter{theorem}{0}

In this section we move from tetration, to the more general concept of hyper-operations. This section is the goal of our paper, and everything before it has been leading up to this result. It shows there exists a sequence of real analytic functions that are recursively equivalent to the hyper-operators defined on the natural numbers. This sequence will contain addition, multiplication and exponentiation.

Hyper-operators express a recursive relationship that, deceptively, is easy to define but becomes much more complicated upon closer analysis. The level of recursion grows for each operator, and takes more and more time to calculate. In the case of natural numbers, these functions grow too fast to evaluate economically, even with a computer. We will stray from what we usually call hyper-operators, however the sequence of functions we construct rightfully still deserve to be called hyper-operators. 

It is quite ironic that hyper-operators are known for their fast growth at infinity and that our extension $\alpha \up^n x$ for $n \ge 2$ is bounded on the positive real line approaching a constant at infinity. Thus a sequence of functions satisfying the recursive pattern of hyper-operators need not necessarily grow unbounded. Its recursion can be satisfied with a sequence of analytic functions in two variables $\alpha \up^n x$ for $\alpha \in (1,e^{1/e})$ and $x \in \mathbb{R}^+$ that sends to $(1,e)$ for $n \ge 2$.

 The hyper-operators are usually defined as binary operators on the natural numbers $\{a+1, a+b, a\cdot b, a \up b = a^b, ...\}$. A sequence, starting from successorship, where each operator is the iterate of the previous operator. In colloquy, addition $a + b$ is iterated succesorship $a + 1...(b\, times)...+ 1$, multiplication $a \up^0 b = a\cdot b$ is iterated addition $a+...(b\, times)...+ a$, exponentiation $a\up b$ is iterated multiplication $a\up^0...(b\, times)...\up^0 a$, tetration $a\up^2 b$ is iterated exponentiation $a\up ...(b\, times)...\up a$, pentation $a\up^3b$ is iterated tetration $a\up^2...(b\, times)...\up^2a$, etc... This idea is more formally stated through a nested recursion, $a\up^n(a \up^{n+1} b)= a\up^{n+1} (b+1)$, and for $n \ge 0$ we have $a \up^n 1 = a$.

Note that the hyperoperators $\up^0, \up$ are defined for complex arguments and continue to satisfy the above recursion. We continue in such a manner and produce for $1 <\alpha < e^{1/e}$ an analytic extension $\alpha \up^n z$ that is holomorphic in $z$ for $\Re(z) > 0$, $n\ge 0$ that satisfies $\alpha \up^n (\alpha \up^{n+1} x) = \alpha \up^{n+1} (x+1)$ for $x \in \mathbb{R}^+$ and $\alpha \up^n 1 = \alpha$.

The hyper-operators are related to the previous sections through a general construction. A sequence of functions $\{f_n(z)\}_{n=0}^\infty$ that satisfy $f_{n+1}(z) = f_n^{\circ z}(1)$ turn out to satisfy the recursion that hyper-operators satisfy. Stated explicitly, $f_n(f_{n+1}(z)) = f_{n+1}(z+1)$. It is not difficult to see then that the problem of analytically continuing $\alpha\up^n z$ is similar to the problem we were just investigating. We want to find a tower of iterates, $f_n$, where the base function $f_0 = \alpha \cdot z$. This type of problem, finding $f_n$, can be solved generally using the techniques we apply below. However we restrict the case to $f_0 = \alpha \cdot z$ and lead by example.

The solution of tetration from Section \ref{sec4} was a solution to the function $f_2(z) = \alpha \up^2 z$. We proceed by induction, noticing all the techniques from Section \ref{sec4} we applied on exponentiation ($f_1(z)$) to solve for its complex iterate can be applied to tetration ($f_2(z)$). Allowing us to find $f_2^{\circ z}(1) = f_3(z) = \alpha \up^3 z$. We continue on, and develop a solution to each function $\alpha \up^n z$ one step at time. In the end we are given a closed form expression for $\alpha \up^n z$ when $1 \le \alpha \le e^{1/e}$ and $\Re(z) > 0$.

The following lemma comes in hand when we attempt to iterate real positive to real positive and monotone growing functions (of which the hyper-operators are). It shows that the iterate of certain real positive to real positive and monotone growing functions is a real positive to real positive and monotone growing function. This allows us to say that the operator $\up \phi(z) = \phi^{\circ z} (1)$ takes monotone functions to monotone functions, a result that expresses more than what we use it for.

\begin{lemma}\label{lm11}
Let $\phi:G \to G$ be a holomorphic function on open $G$. Let $\xi_0 \in \mathbb{R}^+$ be a fixed point of $\phi$ which satisfies $0 < \phi'(\xi_0) < 1$. Further, let $\phi:\mathbb{R}^+ \to \mathbb{R}^+$ and $\phi'(\mathbb{R}^+) \ge 0$ and $(0,\xi_0) \subset \A_{\xi_0}$. Then $\phi^{\circ t}(\xi): \mathbb{R}^+ \times (0,\xi_0)\to (0,\xi_0)$ and $\D{t}{}\phi^{\circ t}\ge 0$.
\end{lemma}

\begin{proof}
Take $\delta > 0$ and observe that $0<\D{\xi}{}\Big{|}_{\xi = \xi_0}\phi^{\circ \delta}(\xi) = \phi'(\xi_0)^\delta < 1$ by evaluating the differintegral. Exchanging the limits and since $\D{\xi}{}\Big{|}_{\xi = \xi_0} \phi^{\circ k}(\xi) = \phi'(\xi_0)^k$ we have $\D{w}{\delta} \Big{|}_{w=0} \sum_{k=0}^\infty \phi'(\xi_0)^k \frac{w^k}{k!} = \phi'(\xi_0)^\delta$. Since $\phi^{\circ \delta}(\mathbb{R}^+) = \mathbb{R}$ and $\phi^{\circ \delta}(\xi_0) = \xi_0$ we can say $\phi^{\circ \delta}(\zeta) \in \mathbb{R}^+$ for $\zeta \in (\xi_0 -\epsilon, \xi_0)$ for some $\epsilon >0$. Further we can say $\D{\zeta}{}\phi^{\circ \delta}(\zeta) >0$ for $\zeta \in (\xi_0 - \epsilon, \xi_0)$ for some $\epsilon>0$. For any $\kappa >0$ there is an $n >N$ such that $\xi \in [\kappa, \xi_0)$ satisfy $\zeta = \phi^{\circ n}(\xi) \in (\xi_0 - \epsilon, \xi_0)$ since the sequence $\phi^{\circ n}(\xi)$ is monotone and converges to $\xi_0$ from below. Otherwise for the first $N$ such that $\phi^{\circ N}(\xi) \ge \phi^{\circ N+1}(\xi)$ we have all $n \ge N$ $\phi^{\circ n}(\xi) \le \phi^{\circ N}(\xi) < \phi^{\circ N}(\xi_0) = \xi_0$ which implies $\phi^{\circ n} (\xi) \not \to \xi_0$. This is in contradiction to the hypothesis of the theorem.  

Now we do a little trick. $\phi^{\circ n}(\phi^{\circ \delta}) = \phi^{\circ \delta}(\phi^{\circ n})$ so that their derivatives satisfy  $\D{\xi}{}\phi^{\circ n}(\phi^{\circ \delta}(\xi))=\D{\xi}{}\phi^{\circ \delta}(\phi^{\circ n}(\xi))= \Big{(}\D{\zeta}{}\phi^{\circ \delta}(\zeta) \Big{)} \phi ' (\phi^{\circ n-1}) \phi' (\phi^{\circ n-2}) \cdots \ge 0$. This tells us performing the operation the other way

$$\D{\xi}{} \phi^{\circ n} (\phi^{\circ \delta}(\xi)) = \phi'(\phi^{\circ n-1}(\phi^{\circ \delta})) \phi'(\phi^{\circ n-2}(\phi^{\circ \delta})) \cdots \D{\xi}{}\phi^{\circ \delta}(\xi) \ge 0$$

and since all the terms on the left are positive or zero and the whole product is positive or zero, $\D{\xi}{}\phi^{\circ \delta}(\xi) \ge 0$ for $\xi \in [\kappa,\xi_0)$ where $\kappa$ is arbitrary. Now $\phi^{\circ k\delta}(\xi)$ is an increasing sequence in $k$ approaching $\xi_0$, implying $\phi^{\circ \delta}(\xi) > \xi$. Otherwise, if it were not increasing $\phi^{\circ \delta}(\xi) \le \xi$ and $\phi^{\circ k\delta}(\xi) \le \phi^{\circ (k-1)\delta}(\xi) < \xi_0$ and $\phi^{\circ k\delta}$ does not approach $\xi_0$. Therefore $\phi^{\circ \delta + t} (\xi) = \phi^{\circ \delta}(\phi^{\circ t}(\xi)) > \phi^{\circ t}(\xi)$ for all $\delta> 0$ and $\xi \in (0,\xi_0)$, therefore $\D{t}{}\phi^{\circ t}(\xi) \ge 0$. It also tells us $0 < \phi^{\circ t}(\xi) < \xi_0$ for all $\xi \in (0,\xi_0)$ and $t \in \mathbb{R}^+$. This gives the result.

\end{proof}

With this we are tempted to say the problem is solved. Now that we know $\alpha \up^n x$ will be monotone and bounded it will have an attracting fixed point $\omega$ such that $\alpha \up^n \omega = \omega$. Since it is attracting and its derivative is greater than or equal to zero we know that $0\le\D{\xi}{}\Big{|}_{\xi = \omega} \alpha \up^n \xi \le 1$. We will have some values $\alpha$ that cannot be iterated. But we fix this by further adding that the values $\alpha$ where $0<\D{\xi}{}\Big{|}_{\xi = \omega} \alpha \up^n \xi < 1$ will be dense in $(1,e^{1/e})$ and everywhere we could not iterate can be arbitrarily approximated uniformly with values $\alpha$ that could be iterated. This devises a solution for $\alpha \up^{n+1} z$ for all $\alpha \in [1,e^{1/e}]$.

\begin{theorem}\label{thm4}
Let $1 \le \alpha \le e^{1/e}$ and $n\ge0$. Define the following holomorphic functions recursively for $z,w \in \mathbb{C}$, $\Re(z) > 0$ and $n \in \mathbb{N}$ with $\alpha \up^0 z = \alpha\cdot z$,
\begin{eqnarray*}
\vartheta_n(w) &=& \sum_{k=0}^\infty \big{(}\alpha\up^{n}\alpha\up^{n}...(k+1)\,times...\up^{n}\alpha\big{)} \frac{w^k}{k!}\\
\alpha \up^{n+1} z &=& \D{w}{z-1}\Big{|}_{w=0} \vartheta_n(w)
\end{eqnarray*}
then,
\begin{enumerate}
\item $\alpha \up^{n} :\mathbb{R}^+ \to \mathbb{R}^+$ and $\alpha \up^n (\alpha \up^{n+1} x) = \alpha \up^{n+1}(x+1)$
\item $\alpha \up^{n} x$ is real analytic in $\alpha$ for $1<\alpha < e^{1/e}$
\item $\D{x}{}\alpha \up^{n} x \ge 0$
\item $\alpha \up^n 0^+ = 1$ for $n\ge 1$
\end{enumerate}
\end{theorem}

\begin{proof}
The proof of this theorem goes by induction. We assume we have a solution for $\alpha \up^n z$ and we will show this admits a solution for $\alpha \up^{n+1} z$. For the cases $\up^0, \up$ the result is shown, arbitrary because each of these are exponentially bounded and can be factored. For $\up^2$ we have shown the result in Section \ref{sec4}, for convenience we will start with $\up^2$ as the base case. We give additional conditions $\alpha \up^n z$ satisfies--for which we will show $\alpha \up^{n+1} z$ satisfies as well. 

We want $\alpha \up^n (x + \delta) > \alpha \up^{n} x$ for $\delta, x \in \mathbb{R}^+$--that the hyper-operators have monotone growth in the second argument. We know this is true for tetration $\up^2$ by Lemma \ref{lm11}, so we have shown the base step of induction. We also know that tetration converges to a point as we move along the real line and that this point is less than or equal to $e$--we will impose this condition on all hyper-operators greater than exponentiation, so that $\alpha \up^n x \to L$ for some positive real number $L\le e$ as $x \to \infty$. We will also assume that $\alpha \up^n x$ is real analytic in $\alpha$; although we have not shown this for tetration $\up^2$ we will give a proof below that suffices. We start by showing we can iterate $\alpha \up^n \xi$ in $\xi$ which will give us the holomorphic candidate $\alpha \up^{n+1} z$, of which we will show satisfies (1) through (4). 

First define the sequence $F_k(x) = \alpha \up^n ( \alpha \up^n (...(k \, times)...(\alpha \up^n x)$ which as $k \to \infty$ we have $F_k(0) \to \omega_\alpha \le e$. This follows because $F_k(0)$ is a bounded and monotone increasing sequence. To show this, $0 < F(0^+) = 1$ so $F(0) < F(F(0)) < F(F(F(0)))<...<F_k(0)$, and $F_k$ is bounded since $F_k(0) < F_k(\infty) = L$. Therefore $F_k(0) \to \omega_\alpha$ as $k \to \infty$.  We know that $\omega_\alpha$ is a fixed point of $F$, since $F(\omega_\alpha) = F (\lim_{k\to \infty} F_k) = \lim_{k\to\infty} F_{k+1} = \omega_\alpha$. Notice that for $0<x < \omega_\alpha$ we have $F_k(0) < F_k(x) < \omega_\alpha$, and therefore, $\lim_{k\to \infty} F_k(x) \to \omega_\alpha$ for $x \in (0,\omega_\alpha)$.

Now we know $0\le F'(\omega_\alpha) = \D{\xi}{}\Big{|}_{\xi = \omega_\alpha}\alpha\up^n\xi \le 1$ through the following argument. The function $\alpha \up^n x$ is increasing on $\mathbb{R}^+$ and therefore its derivative is positive or zero. We know that $\D{\xi}{}\Big{|}_{\xi = \omega_\alpha}\alpha\up^n\xi  \le 1$ because if it were the case that $\D{\xi}{}\Big{|}_{\xi = \omega_\alpha}\alpha\up^n\xi > 1$ then $F_k(x) \not\to \omega_\alpha$--the fixed point would be repelling, it would send a neighborhood about $\omega_\alpha$ outside of itself \cite{ref2}.

We go by cases now, assume $0<F'(\omega_\alpha) < 1$.  We want to show that $1 \in \A_{\omega_\alpha}$. This follows by noting $x \in (0,\omega_\alpha)$ satisfies $\lim_{k \to \infty} F_k(x) = \omega_\alpha$, and there is an open ball about each of these points in the complex plane that satisfy this. We add that $1 \in (0,\omega_\alpha)$ because $\alpha \up^n \alpha \up^n  0 = \alpha$ and $1<\alpha < \omega_\alpha$. This implies $1$ can be connected to $\omega_\alpha$ with open balls whose elements $\xi$ satisfy $\lim_{k\to \infty } F_k(\xi)\to \omega_\alpha$ which implies $1 \in \A_{\omega_\alpha}$. 

Although we do not know if $\A_{\omega_\alpha} \subset \mathbb{C}_{\Re(z)>0}$ our iteration techniques can be applied and we may factor the complex iterates of $F(\xi) = \alpha \up^n \xi$ about one, but we have no knowledge of where $F^{\circ z}(1)$ lives. To visualize what is going on and how we are talking about the function $\up^{n+1}$; imagine again the simply connected set $B$ from Lemma \ref{lma4}. Successive iterations of $\up^n$ $m$ times about $1$ close to this ball and the iterates are well defined here. We state merely that $\alpha\up^{n+1} z$ is an analytic continuation of $\alpha \up^n \alpha \up^n \alpha \up^n ...(m \,times)... \alpha \up^{n+1} z$, where eventually the values live inside of $B$ and we can speak of the domain this function lives in. By analytic continuation this recursion will hold for $z\in\mathbb{R}^+$ since $\alpha \up^n,\alpha \up^{n+1}:\mathbb{R}^+\to\mathbb{R}^+$. To state this rigorously we say the following.

Take $\Re(z) \ge \kappa > 0$ and $m \in \mathbb{N}$ chosen so that $|(\alpha \up^{n+1} (z+m)) - \omega_\alpha| < \epsilon$ which is allowed since $\alpha \up^{n+1}(z+m)$ uniformly tends to $\omega_\alpha$ as $m \to \infty$. Then for $k \in \mathbb{N}$, $\alpha \up^{n} (\alpha \up^{n+1} (k + m)) = \alpha \up^{n+1} (k+m+1)$. This implies, since both functions equal on the natural numbers and they satisfy our factoring bounds, they must equal everywhere. We now have that $\alpha \up^{n} (\alpha \up^{n+1} (z+m)) = \alpha \up^{n+1} (z+m+1)$--where $\Re(\alpha \up^{n+1} (z+m))>0$ since $\alpha \up^{n+1} (z+m)$ is in an $\epsilon$-radius of $\omega_\alpha$ so this is well defined. This implies $\alpha \up^n (\alpha \up^{n+1} z) = \alpha \up^{n+1} (z+1)$ for $\Re(z) > m$. Now by analytic continuation, since $\alpha \up^{n}: \mathbb{R}^+ \to \mathbb{R}^+$ and $\alpha \up^{n+1}: \mathbb{R}^+ \to \mathbb{R}^+$, we know that their composition is well defined and for $x \in \mathbb{R}^+$, $\alpha \up^{n} (\alpha \up^{n+1} x) = \alpha \up^{n+1}( x  +1)$ which shows (1). Since the range and domain that hyper-operators have are rough to analyze, we stick to this as the culminating result--the recursive property is satisfied on the real positive line.

We show (2), that $\alpha \up^{n+1} x$ is real analytic in $\alpha$ by breaking the expression into pieces. Take $h_N(\alpha) = \sum_{k=0}^N \big{(} \alpha \up^n ...(k+1\,times)... \up^n \alpha \big{)} \frac{(-1)^k}{k!(1+k-x)}$ which is analytic in $\alpha$. Fix $x$ and choose $N$ big enough such that $\sum_{k=N+1}^\infty  |\frac{1}{k!(k+1-x)}| < \epsilon/e$. We note that $\alpha \up^{n} \alpha \up^n ... (k \, times)...\up^n \alpha \le e$ which follows because $\alpha \up^n x \to L\le e$ as $x \to \infty$ for $x \in \mathbb{R}^+$ and $\alpha \up^n x$ is monotone in $x$.
\begin{eqnarray*}
|h - h_N| &\le& \sum_{k=N+1}^\infty \big{(} \alpha \up^n ...(k+1\,times)...\up^n \alpha\big{)} \frac{1}{k!|1+k-x|}\\
&\le& e\sum_{k=N+1}^\infty \frac{1}{k!|1+k-x|}\\
&<& \epsilon
\end{eqnarray*}

Now define $p_N(\alpha,w) = \sum_{k=0}^N (\alpha \up^n ...(k+1\,times)...\up^n\alpha) \frac{(-w)^k}{k!}$ analytic in $\alpha$. Let $N$ be chosen such that $\sum_{k=N+1}^\infty \frac{|w|^k}{k!} < \epsilon/e$.

\begin{eqnarray*}
|p - p_N| &<& \sum_{k=N+1}^\infty \big{(} \alpha \up^n ...(k+1\,times)...\up^n \alpha\big{)} \frac{|w|^k}{k!}\\
&<& \sum_{k=N+1}^\infty e \frac{|w|^k}{k!}\\
&<& \epsilon
\end{eqnarray*}

Now take $\int_1^N p(\alpha,w)w^{x-1}\,dw$ which converges uniformly to \\$\int_1^\infty p(\alpha,w)w^{x-1}\,dw$ in $\alpha$. Create the function $g(w) = \sup_{\alpha - \kappa \le q \le \alpha + \kappa} |p(q,w)|$ for some $\kappa > 0$ where $0<\D{\xi}{}\Big{|}_{\xi = \omega_q} q \up^n \xi < 1$. An interval of $q$ that satisfy this is possible because there are only a finite number of $\alpha$ such that $\D{\xi}{}\Big{|}_{\xi = \omega_\alpha} \alpha \up^n \xi = 0,1$ (we will show this below). The function $g$ has a convergent differintegral, take $\int_N^\infty g(w)w^{\sigma-1}\,dw < \epsilon$ and uniform convergence follows.  This implies $\frac{1}{\G(1-x)} (h(\alpha) + \int_1^\infty p(\alpha,w)w^{x-1}\,dw) = \alpha \up^{n+1} x$ is analytic in $\alpha$. We note that $\lim_{x \to \infty} \alpha \up^{n+1} x \to \omega_\alpha < L \le e$

Now we assume that $F'(\omega_\alpha) = \D{\xi}{}\Big{|}_{\xi = \omega_\alpha} \alpha \up^n \xi = 0,1$. We will analyze the function $\omega_\alpha$ in $\alpha$. This function is implicitly defined by the set $$\{x \in \mathbb{R}^+|\alpha \in (1,e^{1/e})\,(\alpha \up^n x) - x = 0\}$$ by the analytic implicit function theorem we have that $\omega_\alpha$ is analytic in $\alpha$ away from its critical points, where ever $\D{x}{}\alpha \up^{n} x \neq 1$. However this function is bounded, so where it is ill-defined a limit will exist but its derivative will blow up. Consequently it will not be analytic at this point, but its limiting value will exist. Now define the set $\mathcal{H} = \{ \alpha \in (1,e^{1/e}) | F'(\omega_\alpha) = 0,1\}$. We know that $\mathcal{H}$ is closed. $\mathcal{H}$ has no limit points by a short proof. If otherwise $F'(\omega_\alpha) = \D{\xi}{}\Big{|}_{\xi = \omega_\alpha}\alpha \up^n \xi = 1,0$ for all $\alpha$ by the identity theorem. This contradicts the analyticity of $\omega_\alpha$ and $\alpha \up^n x$ and the way these were defined.

Now we know that there exists $a_j \in (1,e^{1/e})$ such that  $\lim_{j \to \infty} a_j = \alpha$ and $0 < \D{\xi}{}\Big{|}_{\xi = \omega_{a_j}} a_j \up^n \xi < 1$. We consider the functions $a_j \up^{n+1} z$ which are holomorphic by the first case, and which we show as $j \to \infty$ the limit uniformly converges in $z$. Notice $|a_j \up^{n+1} k - a_i \up^{n+1} k| < \epsilon$ for $j,i>J$ where $J$ can be chosen for all $k$ since $a_j \up^{n+1} k \to \omega_{a_j}$ as $ k \to \infty$ and $|\omega_{a_j} - \omega_{a_i}|<\epsilon$ for big enough $J$. This implies the function $\vartheta_{j}(w) = \sum_{k=0}^\infty a_j \up^{n+1} k \frac{w^k}{k!} \to \vartheta(w)$ uniformly on $\mathbb{C}$.

Let $F_j(y) = a_j \up^n y$. Observe $F_j^{\circ 0}(1) = 1$ and by pulling apart the integral expression and recalling that the transform converges in sectors of $\mathbb{C}$:

$$\lim_{z\to 0} F_j^{\circ z}(1) = 1 = e^{i\theta}\int_0^\infty \vartheta_j(-e^{i\theta} w)\,dw$$

so that as $j \to \infty$ we must have $e^{i\theta} \int_0^\infty \vartheta(-e^{i\theta}w)\,dw = 1$, which implies by Theorem \ref{thmDiff} that $\vartheta(w)$ is differintegrable. This implies $\alpha \up^{n+1} z$ is a holomorphic function for $0 < \Re(z) < 1$. We can analytically continue this function in the same manner as we did in Lemma \ref{lm2}.

We now have $\alpha \up^{n+1} z$, of which the recursive property is satisfied for $x \in \mathbb{R}^+$ since  $\alpha \up^n (\alpha \up^{n+1} x) = \lim_{j \to \infty} a_j \up^n (a_j \up^{n+1} x) =  \lim_{j\to\infty} a_j \up^{n+1} (x+1) = \alpha \up^{n+1} (x+1)$ which shows (1). This function is monotone increasing in $x$ which shows (3), and it is also holomorphic in $z$. The limiting function is analytic in $\alpha$ by the proof we gave above, implying (2) for all $\alpha \in (1,e^{1/e})$. $\lim_{x\to 0^+} \alpha \up^{n+1} x = \lim_{x\to 0^+} \lim_{j \to \infty} a_j \up^{n+1} x = 1$ which shows (4). All this is sufficient to show that hyper-operators at $e^{1/e}$ are well defined, $e^{1/e} \up^n z$ is a holomorphic function in $z$. This implies the hyper-operators are defined for bases $1\le\alpha\le e^{1/e}$. 
\end{proof}

\section{Final Remarks}

We close hoping the reader has seen the connection between our differintegral and iteration. We have tried to be as precise as possible, and hope each theorem was as clear as it needed to be. Looking forward, we ask what other recursive relationships our differintegral can recover. If the reader has cared to notice the generality--what other linear operators can we fractionally iterate with our auxiliary function and fractional calculus? The process of factoring a function by its values on $\mathbb{N}$ that would otherwise be inexpressable provides a solution to some difficult problems; hyper operators and fractional composition being one. We are aware these results extend further than what we have written. In particular we have devised a way, where instead of iterating the linear operator $\mathcal{C}_\phi f = f \circ \phi$, we iterate the linear operator $\bigtriangledown_z f = f(z) - f(z-1)$. This problem is solved using techniques that vary only in subtlety and remain true to using fractional calculus and our auxiliary function $\vartheta$. This problem poses much less of a challenge, however. 

Turning our eyes to iteration particularly, there exists a solution to complex iterates of functions at positive repelling fixed points ($\phi'(\xi_0)> 1$), but the proof requires more work and more complicated theorems from complex dynamics. We do not know whether a solution to iterating functions about fixed points whose derivatives are complex numbers at that fixed point exists using our differintegral $\phi'(\xi_0) \in \mathbb{C}_{|z|<1}$, but we are aware that the method will work in restricted cases.

A problem the author has made much headway on, that for reasons of space was left untouched by this paper, is the iteration of iteration. Phrased less poetically: does there exist an analytic function $\alpha \up^s x$ for $s > 0$? In such a sense we ask whether we can construct $\alpha \up^{1/2} x$ and $\alpha \up^{3/2} x$ such that $\alpha \up^{1/2} (\alpha \up^{3/2} x) = \alpha \up^{3/2} (x+1)$. Do there exist semi-operators between multiplication and exponentiation? The problem requires a whole new approach to the methods developed however and would double the length of this paper. Though not solved, the problem appears to be solvable using fractional calculus and the methods developed in this paper.

We may have glazed over a few of the subjects required in the solution of this problem but we act with the finest care in hoping it was clear. The solution of bounded analytic hyper operators offers many questions and we have quite a few interesting ideas which arise from them alone. This is merely a glimpse of the importance of Ramanujan's master theorem in the area of recursion and iteration. We have found different examples of our methods at work but they are less striking than this one, plus the theory of fractional composition and hyper-operators has always been a subject of great interest. 

\bibliographystyle{amsplain}

\end{document}